\documentclass[12pt, reqno]{amsart}
\usepackage{amssymb}
\def\squarebox#1{\hbox to #1{\hfill\vbox to #1{\vfill}}}

\newcommand{\R}{{\mathbb R}}
\newcommand{\C}{{\mathbb C}}
\newcommand{\N}{{\mathbb N}}

\renewcommand{\Re}{\mathop{\rm Re}\nolimits}
\renewcommand{\Im}{\mathop{\rm Im}\nolimits}

\theoremstyle{plain}

\newtheorem{thm}{Theorem}

\newtheorem{rem}{Remark}

\newtheorem{prop}{Proposition}

\begin{document}
\title[Weyl formula]{Weyl formula for the eigenvalues of the dissipative acoustic operator}

\author[V. Petkov]{Vesselin Petkov}


\address{Universit\'e de Bordeaux, Institut de Math\'ematiques de Bordeaux, 351, Cours de la Lib\'eration,
33405 Talence, France}
\email{petkov@math.u-bordeaux.fr} 
\thanks{e-mail:petkov@math.u-bordeaux.fr}

\numberwithin{equation}{section}

\def\S{{\mathbb S}}
\def\In{\mbox{\rm Int}}
\def\re{\mbox{\rm Re}\:}
\def\im{\mbox{\rm Im}\:}
\def\ts{\tilde{\sigma}}
\def\p{{\mathcal P}}
\def\h{{\mathcal H}}
\def\pa{\partial}
\def\om{\omega}
\def\ep{\epsilon}
\def\curl{{\rm curl}\,}
\def\dive{{\rm div}\,}
\def\grad{{\rm grad}\,}
\def\la{\langle}
\def\ra{\rangle}
\def\ii{{\bf i}}
\def\pa{\partial}
\def\nc{{\mathcal N}}
\def\lg{L^2(\Gamma)}
\def\he{\hat{E}}
\def\hh{\hat{H}}
\def\hc{\hat{c}}
\def\t{tan}
\def\vt{\vert_{tan}}
\def\bvt{\big\vert_{tan}}
\def\hg{H^1_h(\Gamma)}
\def\hn{\hat{n}}
\def\vg{\vert_{\Gamma}}
\def\th{\tilde{h}}
\def\cc{{\mathcal C} }
\def\hd{\la hD\ra}
\def\lc{{\mathcal L}}
\def\rc{{\mathcal R}}
\def\cb{{\bf C}}
\def\bb{{\mathcal B}}
\maketitle

\begin{abstract} We study the wave equation in the exterior of a bounded domain $K$ with dissipative boundary condition $\pa_{\nu} u - \gamma(x) \pa_t u = 0$ on the boundary $\Gamma$ and $\gamma(x) > 0.$ The solutions are described by a contraction semigroup $V(t) = e^{tG}, \: t \geq 0.$ The eigenvalues $\lambda_k$ of $G$ with $\Re \lambda_k < 0$ yield asymptotically disappearing solutions $u(t, x) = e^{\lambda_k t} f(x)$ having exponentially decreasing global energy. We establish a Weyl formula for these eigenvalues in the case $\min_{x\in \Gamma} \gamma(x) > 1.$ For strictly convex obstacles $K$ this formula concerns all eigenvalues of $G.$
\end{abstract} 

{\bf Keywords}: Dissipative boundary conditions, eigenvalues asymptotics \\

{\bf Mathematics Subject Classification 2020}:  35P20, 35P25, 47A40, 58J50

\section{Introduction}

Let $K \subset \R^d,$ $d \geq 2$, be a bounded non-empty domain. Let $\Omega = \R^d \setminus \bar{K}$ be connected. and $K \subset \{x \in \R^d:\:|x| \leq \rho_0\}.$ We suppose that the boundary $\Gamma$ of $K$ is $C^{\infty}.$ 
Consider the boundary problem
\begin{equation} \label{eq:1.1}
\begin{cases} u_{tt} - \Delta_x u = 0 \: {\rm in}\: \R_t^+ \times \Omega,\\
\partial_{\nu}u - \gamma(x) \pa_t u= 0 \: {\rm on} \: \R_t^+ \times \Gamma,\\
u(0, x) = f_1, \: u_t(0, x) = f_2 \end{cases}
\end{equation}
with initial data $(f_1, f_2) \in H^1(\Omega) \times L^2(\Omega) = {\mathcal H}.$
Here $\nu(x)$ is the unit outward normal to $\Gamma$ pointing into $\Omega$ and $\gamma(x)\geq 0$ is a $C^{\infty}$ function on $\Gamma.$ The solution of the problem (\ref{eq:1.1}) is given by $V(t)f = e^{tG} f,\: t \geq 0$, where $V(t)$ is a contraction semi-group in ${\mathcal H}$ whose  generator
$$ G = \Bigl(\begin{matrix} 0 & 1\\ \Delta & 0 \end{matrix} \Bigr)$$
has a domain $D(G)$ which is the closure in the graph norm
$$|\|f\| | = (\|f\|_{{\mathcal H}}^2 + \|G f\|^2_{{\mathcal H}})^{1/2} $$
 of functions $f = (f_1, f_2) \in C_{(0)}^{\infty} (\R^d) \times C_{(0)}^{\infty} (\R^d)$ satisfying the boundary condition $\partial_{\nu} f_1 - \gamma f_2 = 0$ on $\Gamma.$ It is well known  that the spectrum of $G$ in $\Re z < 0$ is formed by isolated eigenvalues with finite multiplicity (see \cite{LP1} for $d$ odd and \cite{P1} for all $d \geq 2$.) Moreover, $G$ has no eigenvalues on the imaginary axis $\ii \R.$ 
Notice that if $Gf =\lambda f$ with $0 \neq f \in D(G), \: \Re \lambda < 0$ and $\partial_{\nu} f_1 - \gamma f_2 =0$ on $\Gamma$, we get 
\begin{equation} \label{eq:1.2}
\begin{cases} (\Delta - \lambda^2) f_1 = 0 \:{\rm in}\: \Omega,\\
\partial_{\nu} f_1 -  \lambda \gamma f_1 = 0\: {\rm on}\: \Gamma \end{cases}
\end{equation}
and $u(t, x) = V(t) f = e^{\lambda t} f(x) $ is a solution of (\ref{eq:1.1}) with exponentially decreasing global energy. Such solutions are called {\bf asymptotically disappearing}. On the other hand, the solutions $u(t, x) = V(t) f$ for which there exists $T > 0$ such that $u(t, x) \equiv 0$ for $t \geq T$ are called {\bf disappearing} (see \cite{Ma1}). For $t_0 > 0$ the closed linear space 
$$H(t_0) = \{ g \in {\mathcal H} :  V(t) g = 0\: {\rm for} \: t \geq t_0\}$$
is invariant under the action of $V(t)$ and if $H(t_0) \neq \{0\},$ then $H(t_0)$ has infinite dimension.
If $H(t_0)$ is not trivial, the scattering system is non controllable (see section 4 in \cite{Ma1} for the definition and details). Majda proved in \cite{Ma1} that for obstacles with analytic boundary $\Gamma$ and analytic $\gamma(x)$  the condition $\gamma(x) \neq 1,
\: \forall x \in \Gamma,$ implies that there are no disappearing solutions.
        
         In this paper in the case $\min_{x \in \Gamma}  \gamma(x) > 1$ we show  that there exists a subspace ${\mathcal H}_{sp} \subsetneq {\mathcal H}$ with infinite dimension generated by eigenfunctions of $G$ such that $V(t) g, \: g \in {\mathcal H}_{sp}$ is  asymptotically disappearing. The eigenvalues $\lambda_k$ sufficiently close to $\R^-$ with $\Re \lambda_k \to - \infty$ present a particular interest for applications since they correspond to solutions decreasing sufficiently fast as $t \to + \infty$.  It is important to know that such eigenvalues exist and to have their asymptotic. It was proved in \cite{CPR} that if we have at least one eigenvalue $
\lambda$ of $G$ with $\re \lambda < 0$, then the wave operators $W_{\pm}$ are not complete, that is ${\text Ran}\: W_{-} \not=  {\text Ran}\: W_{+}$. Hence  we cannot define the scattering operator $S$ related to the Cauchy problem for the free wave equations and the boundary problem (1.1) by the product $ W_{+}^{-1}\circ W_{-}$. When the global energy is conserved in time and the unperturbed and perturbed  problems are associated to unitary groups, the corresponding scattering operator $S(z): L^2(\S^{d-1}) \to L^2(\S^{d-1})$ satisfies the identity
\begin{equation} \label{eq:1.3}
S^{-1}(z)= S^*(\bar{z}),\: z \in \C,
\end{equation} 
providing  $S(z)$ invertible at $z$.  Since $S(z)$ and $S^*(z)$ are analytic in the "physical" half plane $\{z \in \C:\im z < 0\}$ (see \cite{LP}) the above relation implies that $S(z)$ is invertible for $\im z > 0$. For dissipative boundary problems the relation (\ref{eq:1.3}) in general is not true and $S(z_0)$ may have a non trivial kernel for some $z_0, \im z_0 > 0.$ For odd dimensions $d$ Lax and Phillips \cite{LP1} proved that this implies that $\ii z_0$ is an eigenvalue of $G$. Thus the analysis of the eigenvalues of $G$ is important for the location and  the existence of points, where  the kernel of $S(z)$ is not trivial. A similar connection occurs in the analysis of the interior transmission eigenvalues (see \cite{CH} for the definition and more references). More precisely, consider the {\it far-filed operator}
$$(F(k) f) (\theta) = \int_{\S^{d-1} } a(k, \theta, \omega) f(\omega)d\omega,\: (\theta, \:\omega) \in \S^{d-1} \times \S^{d-1}.$$
Here $a(k, \theta, \omega)$ is the scattering amplitude for the Helmholtz equation $(\Delta + k^2 n(x))u = 0, \: x \in K$ with contrast function $n(x) > 0$ and for $d$ odd the scattering operator has the representation
$$S(k) = Id+ \Bigl(\frac{\ii k}{2\pi}\Bigr)^{(d-1)/2} F(k), \: k\in \R.$$ 
Therefore if the kernel of $F(k)$ is non trivial, $k$ is an interior transmission eigenvalue \cite{CH}.

The location in $\C$ of the eigenvalues of $G$ has been studied in \cite{P1} improving previous results of Majda \cite{Ma}. It was proved in \cite{P1} that for the case when $K$ is the unit ball $B_3 =\{x \in \R^3: \;|x|\leq 1\}$ and $\gamma \equiv 1,$ the operator $G$ has no eigenvalues. For this reason we study the cases 
$$(A): \max_{x \in \Gamma}  \gamma(x)  < 1,\:(B): \min_{x \in \Gamma} \gamma(x) > 1.$$ 
The results in \cite{P1} say that in the case (B) for every $0 < \ep \ll 1$ and every $M  \in \N, \: M \geq 1$ the eigenvalues lie in $\Lambda_{ \ep} \cup {\mathcal R}_{N}$, where 
$$\Lambda_{\ep} = \{z \in \C: \: |\Re z | \leq C_{\ep} (1 + |\Im z|^{1/2 + \ep} ), \: \Re z < 0\},$$
$${\mathcal R}_M = \{z \in \C:\: |\Im z | \leq A_M (1 + |\Re z|)^{-M} ,\: \Re z < 0\}.$$
Moreover, for strictly convex obstacles $K$ there exists $R_0 > 0$ such that the eigenvalues lie in ${\mathcal R}_M \cup \{|z| \leq R_0\}.$ In the case (A) the eigenvalues lie in $\Lambda_{\ep}.$ By using the results in \cite{V1}, it is possible to improve the eigenvalue free regions replacing $\Lambda_{\ep}$ by $\{z \in \C:\: -A_0 \leq \Re z < 0\}$ with sufficiently large $A_0 > 0.$

    The existence of eigenvalues has been proved (see Appendix in \cite{P1}) only for the ball $B_3$ and $\gamma \equiv const > 1$  and in this particular case we have 
    \begin{equation} \label{eq:1.4}
   \sigma_p(G) \subset (-\infty, - \frac{1}{\gamma - 1}].
    \end{equation} 
Moreover, we have infinite number of real eigenvalues and as $\gamma \searrow 1$ one gets a large strip $\{z \in \C: \: -\frac{1}{\gamma - 1} < \Re z < 0\}$ without eigenvalues. 

The purpose of this paper is to establish a Weyl formula for the eigenvalues in ${\mathcal R}_M\cap \{z \in \C:\:\Re z< - C_0 \leq -1\}$ in the case (B).  Introduce the set
$$\Lambda = \{ \lambda \in \C:\: |\Im \lambda| \leq C_1( 1 + |\Re \lambda|)^{-2}, \: \Re \lambda\leq -C_0 \leq -1\}$$
containing ${\mathcal R}_M, \: \forall M \geq 2,$ modulo a compact set and denote by $\sigma_p(G)$ the point spectrum of $G$. Increasing  the constant $C_0 > 0$ in the definition of $\Lambda$, we subtract a compact set and this is not important for the asymptotic (\ref{eq:1.5}) below. In the following we assume that $C_0 \geq 2C_1.$ Given $\lambda \in \sigma_p(G)$, we define the algebraic multiplicity of $\lambda$ by
$${\rm mult}\:(\lambda) = {\rm tr}\:\frac{1}{2 \pi \ii} \int_{|z - \lambda| =\ep} (z - G)^{-1} dz$$
with $ 0 < \ep \ll 1$ sufficiently small.
Our main result is the following
\begin{thm} Assume $\gamma(x) > 1$ for all $x \in \Gamma.$ 
Then the counting function of the eigenvalues in $\Lambda$ taken with their multiplicities has the asymptotic
\begin{eqnarray} \label{eq:1.5} 
\sharp \{ \lambda_j \in \sigma_p(G) \cap \Lambda:\: |\lambda_j | \leq r, \: r \geq C_{\gamma}\} \nonumber \\
=  \frac{\omega_{d-1}}{(2 \pi)^{d -1}}\Bigl( \int_{\Gamma} (\gamma^2(x) - 1)^{(d-1)/2} dS_x\Bigr) r^{d-1} + {\mathcal O}_{\gamma} (r^{d-2}),\: r \to \infty,
\end{eqnarray} 
$\omega_{d-1}$ being the volume of the unit ball $\{x \in \R^{d-1}:\: |x| \leq 1\}.$
\end{thm}
The example concerning the ball $B_3$ and (\ref{eq:1.4})  show that the condition $r \geq C_{\gamma}$ is natural since the coefficient before $r^{d-1}$ in (\ref{eq:1.5}) goes to 0 as
$\max_{x \in \Gamma} \gamma(x) \searrow 1.$ Notice that for strictly convex obstacles $K$ in the case (B) we obtain a Weyl formula for {\it all eigenvalues} of $G$.
For Maxwell's equations with dissipative boundary conditions in the particular case $K = B_3, \gamma \equiv const \neq 1,$ the formula (\ref{eq:1.5}) has been obtained in \cite{CP}. Weyl formula for the transmission eigenvalues have been obtained by several authors. We refer to \cite{PV} and \cite{NN} for more references. It is important to note that in \cite{PV} the Weyl formula is established with remainder which depends on the eigenvalue free region. In  \cite{NN} the relation with the eigenvalues free regions is not exploited and the argument is based on a Tauberian theorem which yields a weak remainder. In the present paper we apply the eigenvalue free results in \cite{P1} and the remainder in (\ref{eq:1.5}) is optimal. 

To prove Theorem 1, we apply the approach of \cite{SV} and the construction of a semi-classical parametrix $T(h, z), \: 0 < h \leq h_0, \: z =  -\frac{1}{( 1 + \ii \eta)^2}, \: |\eta| \leq h^2$ for the semi-classical exterior Dirichlet-to-Neumann map $N(h, z)$ given in \cite{V}, \cite{P1}. For $ z = - 1$ the operator $P(h): = T(h, -1) - \gamma(x)$ is self-adjoint and we denote by $\mu_1(h) \leq \mu_2(h) \leq ...$ its eigenvalues counted with their multiplicities. The points $0 < h_k \leq h_0$ for which $\mu_k(h_k) = 0 $ correspond to points $h$ for which $P(h)$ is not invertible. For large fixed $k_0$,  depending on $h_0$, the eigenvalues $\mu_k(h_0)$ are positive, whenever $k > k_0$. Thus if $\mu_k(r^{-1} ) < 0, \: k > k_0,$  we have $\mu_k(h_k) = 0$ for some $r^{-1} < h_k < h_0$ and by a more fine analysis we prove that such a $h_k$ is unique. The operator $P(h)$ can be extended as holomorphic one for complex $\th = h(1 + \ii\eta) \in L$ with $|\eta| \leq h^2$ and $L$ defined in (\ref{eq:2.11}). For the resolvent $(\lambda- G)^{-1}$ a trace formula has been established in \cite{P1} (see Proposition 1). Similarly, a trace formula involving $P^{-1} (\th)$ and  the derivative $\dot{P}(\th)$ can be proved. These two trace formulas differs by negligible terms and this leads to a map between the points $h_k \in L$, where $P(h_k)$ is not invertible and the eigenvalues of $G$. To obtain (\ref{eq:1.5}), one  counts the number of the negative eigenvalues of $P(r^{-1}), \: r \geq C_{\gamma}$ which is given by well known formula. 

    The analysis of the  counting function of the eigenvalues of $G$ lying in a strip $\{z \in \C:\: -A_0 \leq \Re z \leq 0\},\: A_0 > 0,$ as well as the study of the case $(A)$ are open problems. There is a conjecture that there exists a sequence of eigenvalues $\lambda_k,\: |\Im \lambda_k| \to \infty.$  
 For the investigation of these problems it seems  convenient  to use the semi-classical parametrix $T(h, z)$ for the exterior Dirichlet-to-Neumann problem constructed in \cite{S} for strictly  convex obstacles in the hyperbolic region $\{z \in \C: \: z = 1 + \ii h w\},\: |w| \leq B_0. $
    
The paper is organised as follows.  In Section 2 we collect some facts concerning the operator $\cc(\lambda) = \nc(\lambda) - \lambda \gamma$ for $\Re\lambda < 0$, where $\nc(\lambda)$ is exterior Dirichlet-to-Neumann map defined in the beginning of Section 2. We recall a the trace formula involving the resolvent $(G - \lambda)^{-1} $ established in \cite{P1}.  In Section 3 one presents some information for the semi-classical parametrix for $N(h, z)$ and  $z \in Z_{e} = \{ z\in \C: \: z= - \frac{1}{(1 + \ii \eta)^2}  \},\: |\eta| \leq h^2$ based on the construction in \cite{V}, \cite{V2}. The properties of the operator $P(h)$ for $h$ real are treated in Section 4. In Section 5 we compare the trace formulas for $\cc(\lambda)$ and for $P(\th)$ and we prove Theorem 1. Finally, in Section 6 we discuss some generalisations and a dissipative boundary problem for Maxwell's equations.

\section{Preliminaries}

We start with  some facts which are necessary for our exposition (see  \cite{P1}). For $\re \lambda < 0$ introduce the exterior Dirichlet-to-Neumann map
$$\nc(\lambda): H^s(\Gamma) \ni f \longrightarrow \pa_{\nu} u\vert_{\Gamma} \in H^{s-1}(\Gamma),$$
where $u$ is the solution of the problem
\begin{equation} \label{eq:2.1}
\begin{cases} (-\Delta +\lambda^2) u = 0 \: {\rm in}\: \Omega,\: u \in H^2(\Omega),\\
u = f \:{\rm on}\: \Gamma,\\
u :(\ii \lambda)-{\rm outgoing}.
 \end{cases}
\end{equation}
A function $u(x)$  is $(\ii \lambda)$-outgoing if there exists $R > \rho_0$ and $g \in L^2_{comp}(\R^d)$ such that
$$u(x) = (-\Delta_0 + \lambda^2)^{-1} g,\: |x| \geq R,$$
where $R_0(\lambda) =(-\Delta_0 +\lambda^2)^{-1}$ is the outgoing resolvent of the free Laplacian $- \Delta_0$ in $\R^d$ which is analytic in $\C$ for $d$ odd and on the logarithmic covering of $\C$ for $d$ even. The resolvent $R_0(\lambda)$ has kernel
 \begin{equation}  \label{eq:2.2}  R_0(\lambda, x- y) = -\frac{\ii}{4} \Bigl(\frac{-\ii \lambda}{2 \pi |x-y|}\Bigr)^{(n-2)/2}\Bigl( H_{\frac{n-2}{2}}^{(1)}(u)\Bigr) \Bigl\vert_{u = -\ii \lambda |x-y|},
\end{equation} 
 $H_{\nu}^{(1)}(z)$ being the Hankel function of first kind and 
we have the asymptotic 
\begin{equation} \label{eq:2.3}
H^{(1)}_{\nu}(z) = \Bigl(\frac{2}{\pi r}\Bigr)^{1/2} e^{\ii (z - \frac{\nu \pi}{2} - \frac{\pi}{4})} + {\mathcal O}(r^{-3/2}), \: - \pi < {\rm arg}\: z < 2\pi,|z| =r  \to +\infty.
\end{equation} 

The solution of the problem
(\ref{eq:2.1}) with $f \in H^{3/2}(\Gamma)$ has the representation
$$u = e(f) + (-\Delta_D +\lambda^2)^{-1}((\Delta -\lambda^2) (e(f)),$$ 
where $e(f): H^{3/2}(\Gamma) \ni f \to e(f) \in H^{2}_{comp}(\Omega)$ is an extension operator  and $R_D(\lambda) = (-\Delta_D +\lambda^2)^{-1}$ is the outgoing resolvent of the Dirichlet Laplacian $\Delta_D$ in $\Omega$. The cut-off resolvent $R_{\chi} (\lambda)  = \chi(x) R_D(\lambda) \chi(x) $ with $\chi(x) \in C_0^{\infty} (\R^d)$ equal to 1 in a neighbourhood of $K \cup \:{\rm supp}\: \: e(f) $   is analytic for $ \re \lambda < 0$ and meromorphic in $\C$ for $d$ odd and on the logarithmic covering of $\C$ for $d$ even. Consequently, $\nc(\lambda): H^{3/2} (\Gamma) \rightarrow H^{1/2} (\Gamma)$  is  a meromorphic operator-valued function with the same poles as $R_{\chi}(\lambda) .$ The same result holds for the action of $\nc(\lambda)$ on other Sobolev spaces.
Consider the set $\Lambda \subset \{z \in \C:\Re z < -C_0 \leq -1\}$ introduced in Section 1. 
By using the estimates for $R_{\chi}(\lambda)$ for $\Re \lambda< - C_0,$ we obtain 
\begin{equation}\label{eq:2.4} 
\|\nc(\lambda)\|_{H^{1/2}(\Gamma) \to H^{-1/2}(\Gamma)} \leq A_0|\lambda|^2,\: \lambda \in \Lambda.
\end{equation} 

Applying Green's representation for the solution $u(y)$ of (\ref{eq:2.1}) and taking the limit 
$$\Omega \ni y_n \to x \in \Gamma,$$
 we have
$$ (C_{00}(\lambda) f)(x) - (C_{01}(\lambda) \nc(\lambda) f)(x) = \frac{f(x)}{2},\: x \in \Gamma$$
where 
$$(C_{00}(\lambda) f)(x)= \int_{\Gamma} f(y) \frac{\pa}{\pa \nu(y)} R_0(\lambda, x- y)dS_y, $$
$$(C_{01}(\lambda)g)(x) = \int_{\Gamma} g(y)  R_0(\lambda, x- y)dS_y $$
are the Calder\'on operators or double and single layer potentials which have the same analytic properties as $R_0(\lambda, x- y).$
 Melrose showed (\cite{M}, Section 3) that there exists an entire family $P_D(\lambda)$ of compact pseudo-differential operators of order -1 on $\Gamma$ such that
$$-2(-\Delta_{\Gamma} + 1)^{1/2} C_{01}(\lambda) = Id + P_D(\lambda),$$
$\Delta_{\Gamma}$ being the Laplace Beltrami operator on $\Gamma$ equipped with the Riemannian metric induced by the Euclidean one in $\R^d.$ In fact, $-C_{01}(\lambda) $ is a pseudo-differential operator of order -1 with principal symbol $\frac{1}{2}(-\Delta_{\Gamma})^{-1/2}$ (see \cite{M}) and one takes the composition of the operators $\sqrt{-\Delta_{\Gamma} + 1}$ and  $(-\Delta_{\Gamma})^{-1/2}.$ Consequently,  $(Id + P_D(\lambda))^{-1} $ is a meromorphic operator-valued function and   for $\Re \lambda < 0$ one deduces
\begin{equation}\label{eq:2.4bis}
\nc(\lambda) = (Id + P_D(\lambda))^{-1} (-\Delta_{\Gamma} + 1)^{1/2} (Id - 2C_{00}(\lambda)).
\end{equation} 
Since $\nc(\lambda)$ is analytic  for $\Re \lambda < 0$,- 1 is not an eigenvalue of $P_D(\lambda) $ for $\Re \lambda< 0.$ 
On the other hand, $C_{00}(\lambda)$ is a pseudo-differential operator of order -1, hence it is compact one.  
 The Neumann problem 
\begin{equation} \label{eq:2.5}
\begin{cases} (-\Delta + \lambda^2) u = 0 \: {\rm in}\: \Omega,\: u \in H^2(\Omega),\\
\pa_{\nu}u = 0 \:{\rm on}\: \Gamma,\\
u: (\ii \lambda)-{\rm outgoing}. \end{cases}
\end{equation}
 has a non-trivial solution  if the operator $2C_{00}(\lambda)$ has  eigenvalue 1  and this occurs only if $\lambda$ coincides with a resonance $\nu_j, \re \nu_j > 0,$ of the Neumann problem (see \cite{LP}). By Fredholm theorem one deduces that
$$\nc (\lambda)^{-1}  = ( Id - 2C_{00}(\lambda))^{-1} (-\Delta_{\Gamma} + 1)^{-1/2} (Id + P_D(\lambda) ): H^s(\Gamma) \rightarrow H^{s+ 1}(\Gamma)$$
 is meromorphic with poles $\nu_j.$
 
Going back to the problem (\ref{eq:1.2}), for $\Re \lambda < 0$ we write the boundary condition  as follows
$$\cc(\lambda)v: = (\nc(\lambda) - \lambda \gamma) v  = \nc(\lambda) \Bigl( Id - \lambda \nc(\lambda)^{-1}  \gamma \Bigr) v  = 0,\:v=  f_1\in H^{1/2}(\Gamma).$$
Clearly, for $\Re \lambda < 0$ the operator $\cc(\lambda)$ has the same singularities as $\nc(\lambda)$, hence $\cc(\lambda): H^{1/2}(\Gamma) \rightarrow H^{-1/2}(\Gamma)$ is analytic and satisfies the estimate (\ref{eq:2.4}) with another constant $A_0.$ The operator $\nc(\lambda)^{-1}$ is compact and by the results in \cite{P1}  there are points $\lambda_0,\: \re \lambda_0 < 0,$ for which  $Id - \lambda_0 \nc(\lambda_0)^{-1} \gamma$ is invertible. Applying the analytic Fredholm theorem for the operator $ \Bigl( Id - \lambda \nc(\lambda)^{-1}  \gamma \Bigr)$  in the half plane$\Re \lambda < 0$, one concludes that 
\begin{equation} \label{eq:2.6} 
\cc(\lambda)^{-1}  = \Bigl(Id - \lambda \nc (\lambda)^{-1}  \gamma\Bigr)^{-1} \nc (\lambda)^{-1}: H^{-1/2}(\Gamma) \rightarrow H^{1/2}(\Gamma) 
\end{equation} 
 is a meromorphic operator-valued function. Notice that for $\lambda \in \R^{-}$ the operators $\nc(\lambda), \cc(\lambda) $ are self-adjoint. This follows from the Green formula for  $(-\Delta + \lambda^2)$.\\
 \begin{rem} It is important to note that the analyticity of the resolvent $(- \Delta_D + \lambda^2)^{-1} $ for $\Re \lambda < 0$ and the absence of resonances of the Neumann problem in the half plan $\{z \in \C: \Re z < 0\}$ imply that $\cc(\lambda)^{-1} $ is meromorphic for $\Re \lambda < 0$ and $(\ref{eq:2.4bis})$ is not necessary for the proof of this statement.

 \end{rem} 
  
 For the resolvent $(\lambda- G)^{-1}$ in \cite{P1} the following trace formula has been proved.
 
\begin{prop} Let $\delta \subset \{ \lambda \in \C: \: \re 
\lambda < 0\}$ be a closed positively oriented curve without self intersections. Assume that  $\cc(\lambda)^{-1} $ has no poles on $\delta$ . Then
\begin{equation} \label{eq:2.7}
 {\rm tr}_{{\mathcal H}} \: \frac{1}{2 \pi i} \int_{\delta} (\lambda - G)^{-1} d\lambda = {\rm tr}_{H^{1/2}(\Gamma)} \:\frac{1}{2 \pi i} \int_{\delta} \cc(\lambda)^{-1}  \frac{\pa \cc}{\pa \lambda}(\lambda) d \lambda.
\end{equation}
\end{prop}
Since $G$ has only point spectrum in $\re \lambda < 0$, the left hand term in (\ref{eq:2.7}) is equal to the number of the eigenvalues of $G$ in the domain $\omega$ bounded by $\delta$ counted with their algebraic multiplicities. Setting $\tilde{\cc}(\lambda) =  \frac{ \nc(\lambda)}{\lambda} - \gamma,$  we write the right hand side of (\ref{eq:2.7})
 as
\begin{equation} \label{eq:2.8} 
{\rm tr} \frac{1}{2 \pi i} \int_{\delta} \tilde{\cc}(\lambda)^{-1}  \frac{\pa \tilde{\cc}}{\pa \lambda}(\lambda) d \lambda.
\end{equation}

Set $\lambda= -\frac{1}{\th}, \:0 <  \Re \th \ll 1 $ and consider the problem
\begin{equation} \label{eq:2.9}
\begin{cases} (-\th^2\Delta +1) u = 0 \:{\rm in}\: \Omega,\\
-\th \partial_{\nu}u -   \gamma u = 0\: {\rm on}\: \Gamma,\\
u -{\rm outgoing}. \end{cases} 
\end{equation}

We introduce the operator $C(\th) : = -\th \nc(-\th^{-1} ) - \gamma$ and using (\ref{eq:2.8}),  the trace formula (\ref{eq:2.7}) becomes
\begin{equation} \label{eq:2.10}
 {\rm tr}\: \frac{1}{2 \pi i} \int_{\delta} (\lambda - G)^{-1} d\lambda = {\rm tr} \frac{1}{2 \pi i} \int_{\tilde{\delta}} C(\th)^{-1}  \dot{C}(\th) d \th,
\end{equation} 
where $\dot{C}$ denote the derivative with respect to $\th$ and $\tilde{\delta}$ is the curve $\tilde{\delta} = \{z \in \C:  z = -\frac{1}{w}, \: w  \in \delta\}.$\\

Obviously, for $ \lambda \in \Lambda$  one has $|\Im \lambda| \leq 1$ and this implies  $\th \in L$, where
\begin{equation}\label{eq:2.11}
 L : = \{ \th \in \C:\: |\Im \th| \leq C_1 |\th|^{4},\: |\th| \leq C_0^{-1}, \: \Re \th > 0\}.
\end{equation}
We write the points in $L$ as $\th= h( 1 + \ii \eta)$ with $0 < h \leq h_0 \leq C_0^{-1},\: \eta \in \R$. Recall that $\frac{2C_1}{C_0}\leq  1.$ Then $\frac{C_1}{C_0^3} \leq 1/2$ and for $\th \in L$ we get 
$$|\eta|  \leq \frac{1}{2} \sqrt{1 + \eta^2},$$
hence $\eta^2 \leq 1/3.$ This  implies 
$$|\eta| \leq  C_1 h (1 + \eta^2)^2 h^2 \leq h^2,\: h(1 + \ii \eta ) \in L,$$
since $\frac{16 C_1 h}{9}  \leq 1.$
Therefore the problem (\ref{eq:2.9}) becomes
\begin{equation} \label{eq:2.12}
\begin{cases} (-h^2\Delta-z) u = 0 \:{\rm in}\: \Omega,\\
-(1 + \ii \eta)h \partial_{\nu}u -   \gamma u = 0\: {\rm on}\: \Gamma,\\
u -{\rm outgoing}. \end{cases} 
\end{equation}
with $ z = - \frac{1}{(1 + \ii \eta)^2 }=  - 1 + s(\eta), \: |s(\eta)|\leq  (2 + h^2) h^2 \leq3 h^2$.
On the other hand,
$$C(\th) = -(1+ \ii\eta) h \nc(- \th^{-1}) - \gamma(x).$$

\section{Parametrix for $N(h, z)$ in the elliptic region}

In our exposition we will use $h$-pseudo-differential operators
 and we refer to \cite{DS} for more details.
 Let $X$ be a $C^{\infty}$ smooth compact manifold without
boundary with dimension $d-1\geq 1$. Let $(x, \xi)$ be the coordinates in $T^*(X)$ and let $a(x, \xi; h) \in C^{\infty}(T^*(X) \times (0, h_0]).$ 
 Given $\ell, m \in \R$, one denotes by $S^{\ell, m}$ the set of symbols so that 
$$
 |\pa_{x}^{\alpha} \pa_{\xi}^{\beta} a(x, \xi; h)| \leq C_{\alpha, \beta} h^{-\ell} (1 + |\xi|)^{m - |\beta|},\: \forall \alpha, \forall\beta,\quad
  (x, \xi) \in T^*(X).
 $$
If $\ell = 0,$ we denote $S^{\ell, m}$ by $S^{m}.$ 
The $h-$pseudo-differential operator with symbol $a(x, \xi; h)$ is defined by 
$$
(Op_h(a) f)(x) :\:=\
 (2 \pi h)^{-d + 1}\int_{T^*X} e^{\ii \langle x - y, \xi \rangle /h} a(x, \xi; h) f(y) dy d \xi.$$
We define the space of symbols $S_{cl}^{\ell, m}$ which have an asymptotic expansion
$$a(x, \eta; h) \sim \sum_{j = 0}^{\infty} h^{j- \ell} a_j(x, \eta), \: a_j \in S^{m - j}$$ 
and the corresponding classical pseudo-differential operator is given by 
$$(Op(a) f)(x) :\:=\
 (2 \pi)^{-d + 1}\int_{T^*X} e^{\ii \langle x - y, \eta\rangle } a(x, \eta; h) f(y) dy d \eta.
$$
It is clear that by a change of variable $\xi = h \eta$ we may write a $h-$ pseudo-differential operator as a classical one with parameter $h$. We will use this fact in Section 4.
The operators with symbols in $S^{\ell, m}, S_{cl}^{\ell, m}$ are denoted by $L^{\ell, m}, L_{cl}^{\ell, m}$, respectively. The wave front $\widetilde{WF} (A) \subset \widetilde{T^*(\Gamma)} $ of an operator $A \in L^{\ell, m}$ is defined as in \cite{SV}, where $\widetilde{T^*(\Gamma)}$ is the compactification of $T^*(\Gamma).$\\

   We will recall some results for the {\it exterior} semi-classical Dirichlet-to-Neumann map (see \cite{S}, \cite{V}, \cite{P1}). 
   Consider the operator 
$${\mathcal P}(h, z)u = (-h^2 \Delta_x-  z)u, \: z = - 1 + s(\eta).$$ 
In local normal geodesic coordinates $(y_1, y'), y_1 = {\rm dist}\:(y, \Gamma)$ in a neighbourhood ${\mathcal U}$ of $x_0 \in \Gamma$  the operator ${\mathcal P}$ has the form (see \cite{R})
$${\mathcal P}(h,  z) = h^2D_{y_1}^2 + r(y, hD_{y'}) + h^2 q(x)D_{y_1}  - z,\: D_j = -\ii \pa_{y_j}$$
 with $  r(y, \eta') = \la R(y)\eta', \eta'\ra, \:q(y) \in C^{\infty} $.
Here 
$$R(y) = \Bigl\{\sum_{k=1}^d \frac{\pa y_m}{\pa x_k} \frac{\pa y_j}{\pa x_k}\Bigr\}_{ m, j= 2}^d = \Bigl\{\Bigl\la \frac{\pa y_m}{\pa x}, \frac{\pa y_j}{\pa x} \Bigr\ra\Bigr \}_{m, j =2}^d$$
is a symmetric $((d- 1) \times (d-1))$ matrix  and $r(0, y', \eta') = r_0(y', \eta')$, where $r_0(y', \eta')$ is the principal symbol of the Laplace-Beltrami operator $-\Delta_{\Gamma}$ on $\Gamma$ equipped with the Riemannian metric induced by the Euclidean one in $\R^d.$ 
For $ z = -1 + s(\eta)$ introduce 
$\rho(y', \eta',z) =  \sqrt{ z - r_0(y', \eta')} \in C^{\infty}(T^* \Gamma)$
 as the root of the equation 
$$\rho^2 + r_0(y', \eta') -  z = 0$$
  with $ \im \rho(y', \eta', z) > 0.$ 
We have $\rho \in S^1 $ and
 $$\sqrt{- 1 + s(\eta) - r_0} = \ii\sqrt{1 + r_0} - \frac{ s(\eta)}{\sqrt{1 - s(\eta) +r_0} + \ii \sqrt{1+ r_0}}$$ 
 which implies $\rho - \ii\sqrt{1 + r_0} \in S^{-1}.$

 Let  $u$ be the solution of the Dirichlet problem
\begin{equation} \label{eq:3.1}
\begin{cases}  (-h^2 \Delta -  z) u = 0 \: {\rm in}\: \Omega,\\
u = f \: {\rm on} \: \Gamma,\\
 u -{\rm outgoing}. \end{cases} \end{equation}

Consider the semi-classical Sobolev spaces $H_h^k(\Gamma)$ with norm $\|(1 - h^2 \Delta)^{s/2} u\|_{L^2(\Gamma)}$ and
introduce the exterior semi-classical Dirichlet-to-Neumann map
$$N(h, z): H_h^{s}(\Gamma) \ni f \longrightarrow  -  h \pa_{\nu}u\vert_{\Gamma} \in H^{s-1}_h(\Gamma).$$

G. Vodev \cite{V} established for bounded domains $K \subset \R^d, \: d \geq 2,$ with $C^{\infty}$
boundary and solutions $u$ of the Helmoltz equation  $ (- h^2 \Delta  - z) u = 0, \: x \in K,$ an approximation of the interior Dirichlet-to-Neumann map. With some modifications his results can be applied for the exterior Dirichlet-to-Neumann map $N(h, z)$ (see \cite{P1}). We need some information for the parametrix build in \cite{V}, \cite{V2}  in the elliptic region $Z_e: = \{ z \in \C:\: z = -1 + s(\eta)\}.$\\

  For the reader convenience we recall some points of the construction in \cite{V}, \cite{V2}  for $ z \in Z_e.$ Let $\psi \in C_0^{\infty} (U_0),\: \psi = 1$ in a neighbourhood $U_0$ of $x_0 \in \Gamma.$ Denote the local normal geodesic coordinates by $(x_1, x')$ and the dual variables by $(\xi_1, \xi').$ We search a parametrix $u_{\psi}$ of the problem (\ref{eq:3.1}) with boundary data $\psi f$ in the form
$$\tilde{u}_{\psi}(x) = (2 \pi h)^{-d+ 1} \iint e^{\frac{i}{h} ( \varphi(x, \xi', z)+\la y', \xi' \ra)} \phi^2( \frac{x_1}{\delta}) a(x, \xi', h, z) f(y') d\xi'dy'.$$
Here $ 0 < \delta \ll 1$  and $\phi(t) \in C_0^{\infty}(\R)$ is equal to 1 for $|t| \leq 1$ and to 0 for $|t| \geq 2$. We write 
$$R(x) = \sum_{k=0}^{N-1} x_1^k R_k(x') + x_1^N {\mathcal R}_N(x),\: q(x) = \sum_{k=0}^{N-1} x_1 q_k(x') + x_1^N {\mathcal Q}_N (x). $$
For $\varphi$ the eikonal equation modulo $ x_1^N$ becomes
$(\pa_{x_1} \varphi)^2 + \la R(x) \pa_{x'} \varphi, \pa_{x'} \varphi \ra -  z = x_1^N \Phi_N$
and one obtains a smooth solution having the form
$$\varphi = \sum_{k= 0}^N x_1^k \varphi_k (x', \xi', z) ,\: \varphi_0 = - \la x', \xi'\ra,\: \pa_{x_1} \varphi\vert_{x_1 = 0} = \varphi_1 = \rho.$$
The functions $\varphi_k$ satisfy for $ 0 \leq K \leq N-2$ the equalities
\begin{equation} \label{eq:3.2} 
\sum_{k + j = K}(k+ 1) (j + 1) \varphi_{k+ 1} \varphi_{j + 1} + \sum_{k+ j + \ell =K} \la R_{\ell} \nabla _{x'} \varphi_k, \nabla_{x'} \varphi_j \ra-   z= 0.
\end{equation}
Clearly, we can determine $\varphi_{K+1}$ from the above equality since $\rho \neq 0.$  For $ z = - 1$ we have $\rho = \ii \sqrt{1 + r_0}$ and by recurrence one deduces $\varphi_k = \ii \tilde{\varphi}_k$ with real-valued function $\tilde{\varphi}_k$. Thus for $ z = -1$ we have $\varphi = - \la x', \xi'\ra + \ii \tilde{\varphi}$ with real-valued function $\tilde{\varphi}.$ The amplitude of the parametrix has the form
$$a = \sum_{j= 0}^{N-1} h^j a_j(x, \xi', z),  \: a_0\vert_{x_1 = 0}  = \psi ,\: a_j\vert_{x_1 = 0} = 0, \: j \geq 1$$ 
with  $a_j = \sum_{k= 0}^{N} x_1^k a_{k, j}(x', \xi', z),\: a_{0, 0} = \psi ,\: a_{0, j} = 0, \: j \geq 1.$ The functions $a_j$ satisfy the transport equations
$$ 2 \ii \frac{\pa \varphi}{\pa x_1} \frac{\pa a_j} {\pa x_1} +  2 \ii  \la R(x) \nabla_{x'} \varphi, \nabla_{x'}  a_j \ra + \ii (\Delta \varphi) a_j + \Delta a_{j-1}  $$
 $$= x_1^N A_N^{(j)},\: 0 \leq j \leq N - 1,\: a_{-1} = 0.$$
 We write (see Section 3 in \cite{V2}) 
 $$\Delta \varphi = \sum_{k= 0}^{N-1} x_1^k \varphi_k^{\Delta} + x_1^N E_N(x), \: \Delta a_{j-1} = \sum_{k= 0}^{N-1} x_1^k a_{k, j-1}^{\Delta} + x_1^N F_N^{(j-1)}(x)$$
 with 
 $$\varphi_k^{\Delta} = (k+1) (k+2) \varphi_{k+2}+ \sum_{\ell + \nu = k} \Bigl( \la R_{\ell} \nabla_{x'}, \nabla_{x'} \varphi_{\nu}\ra 
+ q_{\ell} (\nu + 1) \varphi_{\nu + 1}  \Bigr),$$ 
$$a_{k, j-1}^{\Delta} = (k+1) (k+2) a_{k+2, j-1}+ \sum_{\ell + \nu = k} \Bigl(  \la R_{\ell} \nabla_{x'}, \nabla_{x'} a_{\nu, j-1} \ra+ q_{\ell}  (\nu + 1) a_{\nu + 1, j-1}  \Bigr).$$
This leads to the equality (see (3.18) in \cite{V2})
\begin{align} \label{eq:3.3}
2 \ii \sum_{k_1 + k_2 = k} (k_1 + 1) (k_2 + 1) \varphi_{k_1 + 1} a_{k_2+ 1, j} + 2 \ii \sum_{k_1 + k_2 + k_3 = k} \la R_{k_1}\nabla_{x'} \varphi_{k_3}, \nabla_{x'} a_{k_3 , j}\ra \nonumber \\ 
+ \sum_{k_1 + k_2 = k} \ii \varphi_{k_1}^{\Delta} a_{k_2, j} = - a_{k, j-1}^{\Delta}\: \:{\rm for}\: 0 \leq k \leq N-1,\: 0 \leq j \leq N-1.
\end{align}
We can determine $a_{k, j}$ by recurrence from the above equality so that $a_{0, 0} = \psi, \: a_{0, j} = 0, \: j \geq 1,\: a_{k, -1} = 0,\: k \geq 0$. Next introduce the operator
$$T_{\psi}(h, z) f = -h\frac{\pa\tilde{u}_{\psi} }{\pa x_1}\vert_{x_1 = 0} = Op_h(\tau_{\psi}) f $$
with 
$$\tau_{\psi} =  -\ii \rho \psi -  \sum_{j= 0}^{N-1} h^{j + 1} a_{1, j} , \: a_{1, j} \in S^{-j}.$$
 By using the outgoing resolvent $(h^2 \Delta_D - z)^{-1}$ for the Dirichlet Laplacian in $\Omega$, we obtain a parametrix $u_{\psi}$ in $\Omega$ and for $z \in Z_{e}$ we have (see Prop. 2.2 in \cite{P1} and \cite{V}) 
 \begin{equation} \label{eq:3.4}
 \|\nc(h, z) (\psi f)- T_{\psi} f\|_{H^N_h(\Gamma)} \leq C_N h^{-s_d + N} \|f\|_{L^2(\Gamma)},\: \forall N \in \N
 \end{equation} 
with $C_N > 0, s_d > 0$ independent of $f, h$ and $z$ and $s_d$ independent of $N$. Taking a partition of unity $\sum_{j= 1}^M \psi_j(x')  \equiv 1$ on $\Gamma,$ we construct a parametrix and define the operator $T(h, z) =\sum_{j=1}^M T_{\psi_j}(h, z).$ For $z = -1$ the symbol $\sqrt{1 + r_0} + \sum_{j = 0}^{N-1} h^{j + 1} a_{1, j}$ of $T(h, -1)$ is real valued 
and we have the estimate (\ref{eq:3.4}) with $T_{\psi}(h, z)$ replaced by $T(h, z).$ Clearly, we may extend the symbol of $T(h, z)$ holomorphically for $\th \in L.$ \\

\section{Properties of the operator $P(h)$}

In this section we assume that $\gamma(x)  > 1,\:\forall x \in \Gamma$ and we study the operator $P(h) = T(h, -1) - \gamma(x) $ when $h$ is real. 
Set
$$\min_{x\in \Gamma} \gamma(x) = c_0 > 1, \: \max_{x \in \Gamma} \gamma(x) = c_1 \geq c_0$$
and  choose a constant $C = \frac{2}{c_1^2}.$ As we mentioned in Section 3, we can consider the operator $P(h)$ as a classical pseudo-differential operator $Op(P)$ with parameter $h$ with classical symbol 
 $P = \sqrt{1 +h^2 r_0}- \gamma + h P_0(x, h\xi) , \: P_0(x, \xi) \in S^0.$ We denote by $(., .)$ the scalar product in $L^2(\Gamma)$ and for two self adjoint operators $L_1, L_2$ the inequality $L_1 \geq L_2$ means $(L_1 u, u) \geq (L_2 u, u),\: \forall u \in L ^2(\Gamma).$
\begin{prop}  Let $ \la h \Delta \ra = ( 1 - h^2 \Delta_{\Gamma})^{1/2}$ and let  $\ep = C (c_0- 1)^2 < 2.$ Then for $h$ sufficiently small we have
\begin{equation}\label{eq:4.1}
h\frac{\pa P(h)}{\pa h} + C P(h)\la h \Delta \ra ^{-1/2} P(h)  \geq \ep(1 - C_2 h) \la h \Delta \ra
\end{equation}
with a constant $C_2 > 0$ independent of $h.$
\end{prop}
\begin{proof} The principal symbol of the operator on the left hand side in (\ref{eq:4.1}) has the form
\begin{align} \label{eq:4.2} 
q_1 = 2h^2  r_0 (1 +h^2 r_0)^{-1/2} + C \sqrt{1 + h^2 r_0} -2 C \gamma(x)+C \gamma^2(x)  ( 1 + h^2 r_0)^{-1/2} \nonumber \\
= (2 + C- \ep) \sqrt{1 + h^2 r_0} -2 C \gamma(x)+(C \gamma^2(x) - 2)  ( 1 + h^2 r_0)^{-1/2}\nonumber \\
 + \ep \sqrt{1 + h^2 r_0} . \end{align} 
Clearly, 
$$( (2 + C- \ep )\la h D\ra u, u) \geq  ((2 + C- \ep)  u, u) $$
and
$$C\gamma^2(x) - 2\leq C c_1^2 - 2 = 0.$$
Therefore,
$$((C \gamma^2(x)  - 2)\la h D \ra^{-1/2} u, u) =  ( (C \gamma^2(x) - 2)(\la hD \ra^{-1/2}-1)  u, u)$$
$$+ ((C \gamma^2(x) - 2)u, u) \geq  ((C \gamma^2(x) - 2)u, u) - hC_1 \|u\|^2, \: 0 < h \leq h_0.$$
Here the operator $(C \gamma^2(x) - 2) (\la hD \ra^{-1/2} - 1) $ has non-negative (classical) principal symbol 
$$\frac{(2 - C\gamma^2(x)) h^2 r_0}{1 + h^2 r_0 + \sqrt{1 + h^2 r_0}}$$
and applied the semi-classical sharp G\"arding inequality (see for instance, \cite{DS}, Theorem 7.12). Taking into account  (\ref{eq:4.2}) and  the inequality  $C( \gamma(x) - 1)^2 - \ep \geq C(c_0 - 1)^2 - \ep = 0,$
one deduces
$$(Op(q_1) u, u) \geq ( (C( \gamma(x) - 1)^2) - \ep)u, u)  + \ep(\la h D \ra u, u)- hC_1 \|u\|^2$$
$$ \geq \ep(\la h D \ra u, u) - h C_1 h \|u\|^2.$$

The full symbol of the operator on the right hand side of (\ref{eq:4.1}) has the form $q_1 + h q_0$. The term
$h (Op(q_0)  u, u)- hC_1   \|u\|^2$ can be absorbed by $\ep C_ 2 h (\la hD \ra u, u)$ taking $\ep C_2 \geq C_1 + \|Op(q_0)\|_{L^2  \to L^2} $
 and this completes the proof.
\end{proof}

\begin{rem} The values of $\ep$ depends on $(c_0 - 1)^2$ and $\ep \searrow 0$ as $c_0 \searrow 1$. In the case when $\gamma \equiv const$ and $K$ is the ball $\{x: \: \|x\| \leq 1\}$ the operator $G$ has no eigenvalues if $\gamma \equiv 1$ $($see \cite{P1}$)$. Moreover, in this case for $\gamma > 1$ the eigenvalues of $G$ lie in the interval $( - \infty, - \frac{1}{\gamma - 1})$. Thus as $\gamma \searrow 1$, in  the domain $\Re \lambda > -\frac{1}{\gamma -1}$ there are no eigenvalues.
\end{rem}
 Next we follow the argument of Section 4, \cite{SV} with some modifications. Consider the semi-classical Sobolev space $H^s(\Gamma)$ with norm $\|u\|_s = \|\la hD \ra^s u \|_{L^2}.$ The operator $P(h) : H^1 \rightarrow L^2$ has derivative $\dot{P}(h) = {\mathcal O}(h^{-1}) : H^1 \rightarrow L^2$. Denote by 
$$\mu_1(h) \leq \mu_2(h) \leq ...\leq \mu_k(h) \leq ...$$ the eigenvalues of $P(h)$ repeated with their multiplicities. \\

   Let $h_1$ be small and let $\mu_k(h_1)$ have multiplicity $m$. For $h$ close to $h_1$ one has exactly $m$ eigenvalues and we denote by $F(h)$ the space spanned by them. We can find a small interval $(\alpha, \beta)$ around $\mu_k(h_1)$, independent on $h$, containing the eigenvalues spanning $F(h)$. Given $h_2 > h_1$ close to $h_1$, consider a normalised eigenfunction $e(h_2)$ with eigenvalue $\mu_k(h_2)$. Let $\pi (h)= E_{(\alpha, \beta)}$ be the spectral projection of  $P(h)$, hence $F(h) = \pi(h) L^2(\Gamma).$  Then $ (\pi(h) - I)\pi(h) = 0$ yields $\pi(h) \dot{\pi}(h)\pi(h) = 0$ and we deduce $\dot{\pi}(h) \vert_{F(h)} = 0.$
   We construct a smooth extension $e(h) \in F(h),\: h \in [h_1, h_2]$ of $e(h_2)$ with $\|e(h)\| = 1,\:\dot{e}(h) \in F(h)^{\perp}$. Obviously, $e(h_1)$ will be normalised eigenfunction with eigenvalue $\mu_k(h_1).$   
    
Considering the eigenvalues $\mu_k(h)$ of $P(h)$ in a small interval $[-\delta, \delta],\: \delta > 0,$ one gets $\|P(h) e(h)\| \leq \delta.$ On the other hand,
$$h \dot{P}(h) = h^2 \Delta \la hD \ra^{-1} + hL_0 = P(h) - \la hD \ra^{-1} + hL_1$$
with  zero order operators $L_0, L_1$ and this implies $|( \dot{P}(h) e(h), e(h))|\leq C_0h^{-1}, \: h \in [h_1, h_2] .$ Therefore
$$|\mu_k(h_2) - \mu_k(h_1)| = \Bigl|\int_{h_1}^{h_2} \frac{d}{dh} ( P(h) e(h), e(h)) dh \Bigr| \leq C_0 \int_{h_1}^{h_2} h^{-1} dh \leq \frac{C_0}{h_1}(h_2 - h_1) .$$
Assuming $\mu_k(h) \in [-\delta, \delta]$, we deduce that $\mu_k(h)$ is locally Lipschitz function in $h$  and its almost defined derivative satisfies $|\frac{\pa \mu_k(h)}{\pa h}| \leq C_0 h^{-1}.$ \\

      To estimate $h \frac{\pa \mu_k(h)}{\pa h}$ from below, we exploit Proposition 2  and apply (\ref{eq:4.1}). For  $h \leq h_0 \leq \frac{1}{8C_2}$ and $\mu_k(h) \in [-\delta, \delta]$ we have
$$h \frac{\pa \mu_k(h)}{\pa h} = ( h \dot{P} (h) e(h), e(h))  \geq \ep(1 - C_2 h)  (\hd e(h), e(h)) - C (\hd^{-1} P(h) e(h), P(h) e(h) ) $$
$$\geq \ep(1 - C_2h)- C \delta^2 \geq \frac{3\ep}{4},$$
choosing 
$$\delta = (c_0 - 1) \sqrt{ \frac{1}{4} - C_2 h_0} \geq \frac{(c_0 -1)} {2\sqrt{2}}.$$

 Consequently, for $h \in [h_1, h_2]$ one has
$$\mu_k(h_2) - \mu_k(h_1) \geq \frac{3\ep}{4}  \int_{h_1}^{h_2} h^{-1}  dh \geq \frac{3 \ep}{4 h_2}   (h_2- h_1)$$
and we obtain
$$\frac{3\ep}{4}   \leq h \frac{d \mu_k(h)}{dh} \leq C_0.$$

Fixing $h_0 > $ small, we conclude that the eigenvalue $\mu_k(h)$ increases when $h$ increases
and $\mu_k(h) \in [-\delta, \delta].$ It is well known (see for instance, \cite{DS}) that 
$$\sharp\{k:  \mu_k(h_0) \leq 0\} = \kappa_0 = \frac{1}{(2 \pi h_0) ^{d-1}} \int _{p_1(x, \xi) \leq 0} dxd\xi +  {\mathcal O}(h_0^{-d + 2}),$$
$p_1(x, \xi)$ being the principal symbol of $P(\Re h)$. Then for $k> \kappa_0$ we have $\mu_k(h_0) > 0$ and if for $h < h_0$ one has $\mu_k(h) < 0$, then there exists a point $h < h_k < h_0$ with the properties $\mu_k(h_k)= 0, \: \mu_k(h) < 0 $ for $0 < h < h_k.$ 
This implies that there exists a sequence $h_{k_0} \geq h_{ k_0 + 1} \geq ...$  of values $0 < h \leq h_0$ such that $\mu_k(h_k) = 0, \: k_0 > \kappa_0.$ These values $h_k$ are precisely those for which $P(h)$ is not invertible. Next we choose $p > d$ and construct the intervals $I_{k, p}$ containing $h_k$ with length $|I_{k, p}| \sim h^{p +1} $ and $|\mu_k(h)| \geq h^p$ for $h \in (0, h_0] \setminus I_{k, p}.$ As in \cite{SV}, one constructs the disjoint intervals $J_{k, p},$ and  we obtain the following
 \begin{prop} [Prop. 4.1, \cite{SV}] Let $p > d$ be fixed. The inverse operator $P(h)^{-1} : L^2 \rightarrow L^2$ exists and has norm ${\mathcal O}(h^{-p})$ for $h \in (0, h_0] \setminus \Omega_p,$ where $\Omega_p$ is a union of disjoint closed intervals $J_{1, p}, J_{2, p},...$ with $|J_{k,p}| = {\mathcal O} (h^{p + 2 - d})$ for $h \in J_{k,p}.$ Moreover, the number of such intervals that intersect $[h/2, h]$ for $0 < h \leq h_0$ is at most ${\mathcal O}(h^{1- p}).$ 
\end{prop} 

\section{Relations between the trace integrals for $C(h)$ and $P(h)$}

In this section we study the operators $C(h)$ and $P(h)$ for complex $h \in L$. We use the notation $h$ instead of $\th$ used in Sections 2, 3. For $ z = - 1$ the operator $T(\Re h, -1)$ constructed in Section 3 has principal semi-classical symbol $\sqrt{1 + r_0}$, so it is elliptic.  The ellipticity holds also for the operator $T(h, z),\: h \in L, z = - 1 +  s(\eta) $, holomorphic with respect to $h$, provided $|h|$ small enough.  On the other hand,
$P( h) = ( 1 + \ii \eta)T(h,z) - \gamma(x)$ is not elliptic and for $h \in \R, \:\eta = 0,\:\: z = -1$ its semi-classical principal symbol vanishes on the set
$$\Sigma = \{(x, \xi) \in T^*(\Gamma):\: r_0(x, \xi) = \gamma^2 - 1\}.$$
For the symbol $r_0(x, \xi)$ of the Laplace-Beltrami operator on $\Gamma$ there exists a constant $C_3 > 0$ such that $r_0(x, \xi) \geq C_3\|\xi\|^2,\: (x, \xi) \in T^*(\Gamma).$
Choose a constant $B_0 > 0$ so that $\sqrt{C_3} B_0 \geq 2 c_1$  and consider a  symbol $\chi(x, \xi) \in C^{\infty}_0 (T^*(\Gamma) ),\: 0 \leq \chi(x, \xi) \leq 2$ such that
$$\chi(x, \xi) = \begin{cases} 2,\:\: x \in \Gamma,\: \|\xi\| \leq B_0,\\
0,\:\: x \in \Gamma,\:\|\xi\| \geq B_0 +1.\end{cases}$$
Introduce the operator 
$$\tilde{M}(h) = P(\Re h) + \gamma(x) \chi(x, hD_x) =  T(\Re h, -1) + \gamma(x) (\chi(x, hD_x) - 1). $$
The principal symbol of $\tilde{M}(h)$ has the form
$$\tilde{m}(x, \xi) =  \sqrt{1 + r_0} + \gamma(x) (\chi(x, \xi) - 1).$$
Clearly, $\tilde{M}(h)$ is elliptic since for $\|\xi\| \leq B_0$ one gets $\Re \tilde{m}(x,\xi) \geq c_0,$ while for $\|\xi\| > B_0$ we have
$$|\tilde{m}(x, \xi)| \geq \sqrt{C_3}\|\xi\| - c_1 \geq \frac{\sqrt{C_3}}{2} \|\xi\|  + \frac{\sqrt{C_3}}{2} B_0 - c_1 \geq  \frac{\sqrt{C_3}}{2} \|\xi\|.$$
Consequently, $\tilde{m}(x,\xi) \in S^{1}_0$, the operator $\tilde{M} (h)^{-1}: H^s - H^{s+1}$ is bounded by ${\mathcal O}_s(1)$
and $\widetilde{WF} (P(\Re h) - \tilde{M}(h)) \cap \{\|\xi\| \gg B_0 + 1\} = \emptyset .$ Since $\chi(x, \xi)$ vanishes for $\|\xi\| \geq B_0 + 1$, by applying Proposition A.1 
in \cite{SV}, we can extend holomorphically $\chi(x, h D_x)$ to $\eta(x, \th D_x)$ in the domain $L$. As we mentioned in Section 3,  the operator $P(h)$ also has a holomorphic extension for  $\th \in L$.Thus
$\tilde{M}(h)$ has a holomorphic extension 
$$M(h) =P(h) + \gamma(x) (\eta(x, \th D_x) - 1)$$
 for $\th \in L$ and $\widetilde{WF} (P(h) - M(h)) \cap \{\|\xi\| \gg B_0 + 1\} = \emptyset .$ The last relation implies $P(h) - M(h) : {\mathcal O}(1): H^{-s} \rightarrow H^s, \: \forall s.$

Now we can repeat without any change the proof of Lemma 5.1 in \cite{SV}, exploiting Proposition 2. First we obtain 
\begin{equation} \label{eq:5.1}
\|P(h)^{-1} \|_{\lc(H^{-1/2}, H^{1/2})} \leq C \frac{\Re h}{|\Im h|}, \: \Re h > 0, \: \Im h \neq 0.
\end{equation} 
Next, one deduces the estimate
\begin{equation}\label{eq:5.2} 
\|P(h)^{-1} \|_{\lc(H^s, H^{s+1})} \leq C_s \frac{\Re h}{|\Im h|}, \: \Re h > 0, \: \Im h \neq 0
\end{equation}
applying (\ref{eq:5.1}) and the representation
$$P^{-1} = M^{-1} - M^{-1} (P- M)M^{-1} + M^{-1}(P - M) P^{-1} (P- M) M^{-1},$$ 
combined with the property of $P(h) - M(h)$ mentioned above. 
Following \cite{SV}, introduce a piecewise smooth simply positively oriented curve $\gamma_{k.p}$ as a union of four segments: $\Re h \in J_{k, p},\: \Im h = \pm (\Re h)^{p+1}$ and $\Re h \in \pa J_{k, p},\: |\Im h|\leq (\Re h)^{p+ 1},$ where $J_{k, p}$ is one of the intervals in $\Omega_p$ defined in Proposition 3. Then  we have
\begin{prop}[Prop. 5.2, \cite{SV}] For every $h \in \gamma_{k, p}$ the inverse operator $P(h)^{-1}$ exists and
$$\|P(h)^{-1} \|_{\lc(H^s, H^{s+1})} \leq C_s (\Re h)^{-p},\: h \in \gamma_{k, p}.$$
\end{prop}
To estimate $C(h)^{-1}$, we write
$$C(h) = - (1 + \ii \eta) hN(\Re h, z) - \gamma(x) =  ( 1 + \ii \eta) T(\Re h, z) - \gamma(x) + \rc_m(\Re h,z)$$ 
$$= P(h)+ \rc_m(h, z),\: m \gg 2p$$
with $\rc_m(h, z) : {\mathcal O}((\Re h)^m): H^s \to H^{s + m-1}. $ Therefore
\begin{equation} \label{eq:5.3}
C(h) P(h)^{-1} = Id +  \rc_m(\Re h, z) P( h)^{-1} 
\end{equation} 
 and Proposition 4 imply
$$\Bigl\| \rc_m(\Re  h, z) P(h)^{-1}\Bigr \|_{\lc(H^s, H^{s+m})} \leq C_s (\Re h)^{-p + m}.$$
  For small $\Re h$ the operator on the right hand side of (\ref{eq:5.3}) is invertible and
$$C(h) P(h)^{-1} \Bigl  (Id  + \rc_m(\Re h, z) P(h)^{-1} \Bigr)^{-1}  = Id .$$
On the other hand, the operator $C(h)$ is elliptic for $|\xi| \gg 1$ and  this implies that $C(h): H^{1/2} \rightarrow H^{-1/2}$ is a Fredholm operator. The index of $C(h)$ is constant for $h \in L$ 
and according to the results in \cite{P1}, this index is 0. Hence the right inverse to $C(h)$ is also a left inverse, so it is two side inverse. Thus we obtain
\begin{equation} \label{eq:5.4}
\|C(h)^{-1}\|_{\lc(H^s, H^{s+ 1})} \leq C_s (\Re h)^{-p}, \: h \in \gamma_{k, p}.
\end{equation}
Moreover,
\begin{equation}\label{eq:5.5} 
C(h)^{-1} - P(h)^{-1} = P(h)^{-1} \Bigl( \Bigr (Id  + \rc_m(\Re h, z) P(h)^{-1} \Bigr)^{-1} - Id\Bigr) = K(h)
\end{equation}
with $K(h) = {\mathcal O}_s(|h|^{m- 2p}) : H^{s} \rightarrow H^{s + m + 1}, \: \forall s, \: h \in \gamma_{k, p}.$ To estimate $\dot{C}(h) - \dot{P}(h),$ notice that
$C(h) - P(h)$ is holomorphic with respect to $h$ in $L_0$ and by Cauchy formula
$$\dot{C}(h)- \dot{P}(h) = \frac{1}{2 \pi \ii}\int_{\tilde{\gamma}_{k, p}  } \frac{ C(\zeta) - P(\zeta)} {\zeta - h} d\zeta= \frac{1}{2 \pi \ii}\int_{\tilde{\gamma}_{k, p} } \frac{\rc_m(\Re h, z)} {\zeta - h} d\zeta= K'(h),$$
where $\tilde{\gamma}_{k,p}$ is the boundary of a domain containing $\gamma_{k, p}$ with the property ${\rm dist}\: (\tilde{\gamma}_{k,p}, \gamma_{k. p}) 
\geq  (\Re h)^p. $ Thus yileds
$$K'(h) =  {\mathcal O}_s(|h|^{m- p}) : H^{s} \rightarrow H^{s + m + 1}, \: \forall s, \: h \in \gamma_{k, p}.$$

Concerning the operator $P(h)$, we obtain a trace formula repeating without any change the argument in \cite{SV}. Let $\mu_k(h_k) = 0, \: k \geq k_0.$
  It is easy to see that $\mu_k(h)$ has no other zeros for $0 < h \leq h_0,$ exploiting the fact that $\mu_k(h)$ in increasing for
 $\mu_k(h) \in [-\delta, \delta].$
 One defines the multiplicity of $h_k$ as the  multiplicity of the eigenvalue $\mu_k(h_k)$. Then  we have
\begin{prop} [Prop. 5.3, \cite{SV}] Let $\beta  \subset L$ be a closed positively oriented $C^1$ curve without self intersections which avoids the points $h_k$ with $\mu_k(h_k) = 0.$ Then 
  $${\rm tr} \frac{1}{2 \pi \ii}  \int_{\beta} P(h)^{-1} \dot{P}(h) dh$$
  is equal to the number of $h_k$ in the domain bounded by $\beta.$   
\end{prop}
Now we may compare the trace formula for $C(h)$ and $P(h)$. First we  compare the integrals over $\gamma_{k, p}.$ We have
$${\rm tr}\: \frac{1}{ 2 \pi \ii} \int_{\gamma_{k,p}} C(h)^{-1} \dot{C}(h) dh = {\rm tr}\: \frac{1}{ 2 \pi \ii} \int_{\gamma_{k,p}} (C(h)^{-1} -P(h)^{-1} )\dot{C}(h) dh $$
$$+ {\rm tr}\: \frac{1}{ 2 \pi \ii} \int_{\gamma_{k,p}} P(h)^{-1} \dot{C}(h) dh=  {\rm tr}\: \frac{1}{ 2 \pi \ii} \int_{\gamma_{k,p}} P(h)^{-1} \dot{C}(h) dh
+ {\mathcal O}_p((\Re h)^{m-2 p}).$$
Here we have used (\ref{eq:5.5})  and the estimate
$$\|\dot{C}(h)\|_{\lc(H^{1/2}, H^{-1/2})} \leq C |h|^{-2}, \: h \in L.$$
which follows from (\ref{eq:2.4}).  Next the property of $K'(h)$ yields
$$ {\rm tr}\: \frac{1}{ 2 \pi \ii} \int_{\gamma_{k,p}} P(h)^{-1} \dot{C}(h) dh = {\rm tr}\: \frac{1}{ 2 \pi \ii} \int_{\gamma_{k,p}} P(h)^{-1} \dot{P}(h) dh + {\mathcal O}_p((\Re h)^{m- 2p}).$$

For small $h$ and $m \gg 2p$ the terms ${\mathcal O}_p((\Re h)^{m- 2p})$ are negligible and we obtain, as in \cite{SV}, a map $\ell_p$ between the set of points  $h_k \in (0, h(p)]$
counted with their multiplicities and the eigenvalues $\ell_p(h_k) \in \Lambda $ counted with their multiplicities. The number of points $h_k \in J_{k, p}$ counted with their multiplicities is equal to the number of eigenvalues $\lambda_j= \ell(h_k)$ of $G$ counted with their multiplicities lying in $\Lambda_{k, p} = \{ z \in \C: z = - \frac{1}{\zeta},\: \zeta \in \omega_{k, p}\}$,
$\omega_{k, p} \subset L$ being the domain bounded by $\gamma_{k, p}.$ Notice that for a point $h_k$ we could have many $\lambda_j \in \ell_p(h_k) \subset \Lambda_{k,p}.$ On the other hand, for every $ \lambda_j \in \ell_p(h_k)$ one has
$$|\lambda_j +\frac{1}{h_k} | \leq C_p h_k^{p + 2 - d} .$$
The integral over $\beta$ in Proposition 5 can be presented as a sum of integrals over $\gamma_{k, p}$ plus integrals over curves $\alpha_{j, p}$ which are the boundary of domains $\beta_{j,p}$ such that $\beta_{j, p} \cap \Omega_p = \emptyset,\: \forall j.$ By Proposition 3 for $h \in (0, h_0] \setminus \Omega_p$ the operator $P(h)$ is invertible. Applying an argument similar to that used above, one concludes that $P(h)$ is invertible for $h \in \beta_{j, p}$. Consequently, there are no contributions from the integrals over $\alpha_{j, p}$ and we must sum the contributions over the  integrals over $\gamma_{k, p},$ that is the sum of the number of the corresponding points $h_k.$\\

Consider the counting function 
$${\bf N}(r) = \sharp \{\lambda \in \sigma_{p}(G)\cap \Lambda:\: |\lambda | \leq r, \:\Re \lambda \leq -C_0\},\: r > C_0$$
with $h_0^{-1} = C_0 > 0$ large enough.
Then for $|\lambda| \leq r$ we must consider  $r^{-1} < |\th|,\: \th \in L.$   Modulo a finite number eigenvalues (see Section 4 and the number $\kappa_0$), we are going to count the points $r^{-1}  <h_k  \leq h_0$ and the number of the eigenvalues $\mu_k(h)$ for which we have $\mu_k(h_k) = 0.$ Hence we have $\mu_k(r^{-1} ) < 0,$ since otherwise we obtain a contradiction.  The problem is reduced to find the number of the negative eigenvalues of $P(r^{-1})$ which is given by well known formula
$$\frac{r^{d-1}}{(2 \pi)^{d- 1} } \int_{p_1( x, \xi)  \leq 0} dx d\xi + {\mathcal O}_{\gamma} (r^{d-2}).$$
Clearly, 
    $$\int_{p_1( x, \xi)  \leq  0} dx d\xi = \int_{r_0(x, \xi) \leq \gamma^2(x) - 1} dx d\xi = \int_{\Gamma} (\gamma^2(x) - 1)^{(d-1)/2}( \int_{r_0(x, \eta) \leq 1} d\eta) dx.$$
    For the induced  Riemannian metric on $\Gamma$ the integral over the dual variable $\eta$ yields  the volume $\omega_{d-1} $ of the unit ball $\{x \in \R^{d-1}:\:|x| 
   \leq 1\}$ and we obtain the asymptotic (\ref{eq:1.5}). This completes the proof of Theorem 1.
   
   \section{Generalisations} 
 
We may study with some modifications a more general dissipative boundary problem
\begin{equation} \label{eq:6.1}
\begin{cases} u_{tt} - \Delta_x u + c(x) u_t= 0 \: {\rm in}\: \R_t^+ \times \Omega,\\
\partial_{\nu}u - \gamma(x) \pa_t u - \sigma(x) u= 0 \: {\rm on} \: \R_t^+ \times \Gamma,\\
u(0, x) = f_1, \: u_t(0, x) = f_2, \end{cases}
\end{equation}
 where $c(x) \geq 0,\: \sigma(x) \geq 0 $ are smooth functions defined respectively in $\R^d$ and $\Gamma$  and $c(x) = 0$ for $|x| \geq R_0 > 0$ (see \cite{Ma1}). The solution is given by a semi-group $V(t)= e^{tG}, \: t \geq 0$ with $f =(f_1, f_2)$ in the energy space $\mathcal H_E$ with norm
 
 $$\|f\|^2_{{\mathcal H}_E} = \int_{\Omega} ( |\nabla_x f_1|^2 + |f_2|^2) dx + \int_{\Gamma} \sigma |f_1|^2 dy. $$ 
 The generator of $V(t)$ has the form
$$ G = \Bigl(\begin{matrix} 0 & 1\\ \Delta & c \end{matrix} \Bigr)$$
with a domain $D(G)$ being the closure in the graph norm
$$|\|f\| |_E = (\|f\|_{{\mathcal H}_E}^2 + \|G f\|^2_{{\mathcal H}_E})^{1/2} $$
 of functions $f = (f_1, f_2) \in C_{(0)}^{\infty} (\R^d) \times C_{(0)}^{\infty} (\R^d)$ satisfying the boundary condition $\partial_{\nu} f_1 - \gamma f_2 - \sigma f_1= 0$ on $\Gamma.$ 
 If we have an eigenfunction $f = (f_1, f_2)$ with $Gf = \lambda f,$ and $\lambda = -\frac{1}{\th}$ (for simplicity we keep the notation of Section 2), then $u = f_1$ is a solution of the problem
 \begin{equation} \label{eq:6.2}
\begin{cases} (-\th^2\Delta +1 - \th c) u = 0 \:{\rm in}\: \Omega,\\
-\th \partial_{\nu}u -   \gamma u + \th \sigma u= 0\: {\rm on}\: \Gamma,\\
u -{\rm outgoing}. \end{cases} 
\end{equation} 
 Therefore with $\th = h( 1 + \ii \eta), \: \eta \in \R, z = - \frac{1}{(1 + \ii \eta)^2}$ we obtain the problem 
 \begin{equation}
 \begin{cases} (-h^2\Delta -z - \frac{h}{1 + \ii \eta}  c) u = 0 \:{\rm in}\: \Omega,\\
-( 1 + \ii \eta) h\partial_{\nu}u -   \gamma u + h( 1 + \ii \eta) \sigma u= 0\: {\rm on}\: \Gamma,\\
u -{\rm outgoing}. \end{cases} 
\end{equation}  
We need to consider the semi-classical exterior Dirichlet-to-Neumann operator $N(h, z)$ related to the operator $ -h^2 \Delta - z - \frac{h}{1 + \ii \eta} c.$ The construction of the semi-classical paramterix for $N(h, z)$ is the same as in \cite{V}, \cite{P1} . The term $\frac{h}{1 + \ii \eta} c$ is lower order operator and the principal symbol of $N(h, z)$ is $\sqrt{ 1 + r_0}$. Next we deal with the operator
$$P(\th)  = (1 + \ii \eta) N(h, z) - \gamma - h ( 1 + \ii \eta) \sigma$$
and the self-adjoint operator $P(h) = N(h, z) - \gamma - h \sigma.$
Hear $h( 1 + \ii \eta) \sigma$ is a lower order operator and we may repeat the arguments of Sections 4, 5. Under the assumptions of Theorem 1 one obtains a Weyl formula (\ref{eq:1.5}) with the same leading term. We leave the details to the reader.\\

         We hope that our arguments combined with the construction of a semi-classical parametrix in \cite{V3} can be applied for the analysis of the eigenvalues of  Maxwell's equations with dissipative boundary conditions
         \begin{equation} \label{eq:6.4}
 \begin{cases} \pa_t E - \curl H = 0,\: \pa_t H  +\curl E = 0 \: {\rm in} \: \R_t^+ \times \Omega,\\
 \nu \wedge E - \gamma(x)  (\nu \wedge \nu \wedge H) = 0 \: {\rm on}\: \R_t^+ \times \Gamma, \\
 E(0, x) = E_0(x), \: H(0, x) = H_0(x),
 \end{cases}
 \end{equation}  
 where $d = 3,\: (E_0, H_0 ) \in L^2(\R_t^+ \times \Omega: \C^6),\: \gamma(x) > 0, \: \forall x \in \Gamma$. The solution of  (\ref{eq:6.4}) is given by a contraction semi-group    
 $V_b(t)   = e^{tG_b},\: t \geq 0$ (see \cite{CP} for the definition of $G_b$) and the spectrum of $G_b$ in the half-plan $\{ z \in \C: \Re z < 0\}$ is formed by isolated eigenvalues with finite multiplicities \cite{CPR}.    
 
      We sketch briefly below the similitudes with the analysis in Section 2. If $(E, H) \neq 0$ is an eigenfunction of $G_b$ with eigenvalue $\lambda$, then
      \begin{equation} \label{eq:6.5}
      \begin{cases} \curl E = - \lambda H,\: \curl H = \lambda E\:  {\rm in}\:\Omega,\\
      \frac{1}{\gamma(x)} \Bigl(\nu \wedge \nu \wedge E \Bigr)+ \nu \wedge H = 0 \: {\rm on} \: \Gamma,\\
      (E, H) : (\ii \lambda)-outgoing.
      \end{cases}
      \end{equation}         
 Consider the problem
       \begin{equation} \label{eq:6.6}
      \begin{cases} \curl E = -\lambda H,\: \curl H = \lambda E\:  {\rm in}\:\Omega,\\
      \nu \wedge E= f\: {\rm on} \: \Gamma,\\
      (E, H) : (\ii \lambda)-outgoing.
      \end{cases}
      \end{equation}  
 In the space ${\mathcal H}_s^t (\Gamma) := \{u \in H^s(\Gamma):\: \langle \nu, u\rangle = 0\}$ introduce the operator $\nc_b(\lambda) : H_{s+1}^t (\Gamma)\rightarrow H_s^t(\Gamma)$ defined by
 $$\nc_b(\lambda) f = \nu \wedge H\vert_{\Gamma},$$
 $(E, H)$ being the solution of the problem (\ref{eq:6.6}).
  The operator $\nc_b(\lambda)$ is the analog of the exterior Dirichlet-to-Neumann operator in Section 2 (see \cite{V3}) and the boundary condition on (\ref{eq:6.5}) can be written as
  $$\cc_b(\lambda) f = \frac{1}{\gamma(x)} (\nu \wedge f) + \nc_b(\lambda) f = 0,\: f = \nu \wedge E\vert_{\Gamma}.$$      

The outgoing resolvent of the problem 
 \begin{equation*} 
      \begin{cases} \curl E = -\lambda H + F_1,\: \curl H = \lambda E + F_2\:  {\rm in}\:\Omega,\\
      \nu \wedge E= 0\: {\rm on} \: \Gamma,\\
      (E, H) : (\ii \lambda) -outgoing,
      \end{cases}
      \end{equation*} 
  is analytic  for $\Re \lambda < 0$ since the above problem corresponds to a self-adjoint operator.  
  Therefor we can prove that $\cc_b(\lambda)$ is analytic for $\Re \lambda < 0.$ In the same way from the fact that for $\Re \lambda < 0$ there are no non trivial solutions of the problem   
  \begin{equation*}
     \begin{cases} \curl E = -\lambda H,\: \curl H = \lambda E\:  {\rm in}\:\Omega,\\
      \nu \wedge H= 0\: {\rm on} \: \Gamma,\\
      (E, H) : (\ii \lambda)-outgoing,
      \end{cases}
      \end{equation*}  
  one concludes that $\nc_b(\lambda)^{-1} $ is analytic for $\Re \lambda < 0$. As in Section 2, one deduces that $\cc_b(\lambda)^{-1} $ is a meromorphic operator valued function for $\Re \lambda < 0$ (see (\ref{eq:2.6}) and Remark 1).  Assuming $\gamma(x) \neq 1,\: \forall x \in \Gamma,$ according to the results in \cite{CPR1}, for every $\ep > 0$ and every $M \in \N, \: M \geq 1$ the eigenvalues of $G_b$ lie in $\Lambda_{\ep} \cup {\mathcal R}_M.$ Next one can establish a trace formula involving $(\lambda - G_b)^{-1}$ and for the analysis of the counting function of the eigenvalue of $G_b$  in $\Lambda$
  it is possible to apply the strategy of Sections 4, 5 combined with the semi-classical parametrix constructed in \cite{V3}. 
      
   \section*{Acknowledgments} Thanks are due to the referees for their critical comments and useful suggestions leading to an improvement  of  the previous version of the paper.

\end{document}